\documentclass[12pt]{article}
\usepackage{amsfonts}
\usepackage{amssymb}
\usepackage{amsmath}
\usepackage{amsthm}
\usepackage{amsfonts}
\usepackage[T1]{fontenc}
\usepackage[utf8]{inputenc} 
\usepackage{cite} 
\usepackage{stmaryrd}
\usepackage[lflt]{floatflt}
\usepackage{graphicx}
\usepackage{fancyhdr}
\usepackage{color}
\usepackage[affil-sl]{authblk}
\usepackage{cancel}
\usepackage{subfigure}
\usepackage{makeidx}
\usepackage{latexsym}

\usepackage{epsfig}
\usepackage{stmaryrd}
\usepackage{float}
\usepackage{tabularx}
\usepackage{verbatim}
\usepackage{color}
\usepackage{verbatim}
\usepackage{enumerate}
\usepackage[usenames,dvipsnames]{xcolor}

\usepackage{geometry}

\newtheorem{teo}{Theorem}
\newtheorem{conj}{Conjecture}

\newtheorem{propo}{Proposition}
\newtheorem{lema}{Lemma}
\newtheorem{coro}{Corollary}
\newtheorem{obs}{Remark}

\newcommand{\dd}{{\rm d}}

\newcommand{\bbR}{{\mathbb R}}

\newcommand{\al}{\alpha}
\newcommand{\la}{\lambda}
\newcommand{\p}{\partial}
\newcommand{\be}{\beta}
\newcommand{\G}{\Gamma}
\newcommand{\T}{\Theta}

\newcommand{\Wlams}{W\left(-\frac{x}{\la_s t^{\al/2}},-\frac{\al}{2},1\right)}
\newcommand{\Wlaml}{W\left(-\frac{x}{\la_l t^{\al/2}},-\frac{\al}{2},1\right)}

\newcommand{\varsims}{\frac{x}{\la_s t^{\al/2}}}

\begin{document}

\begin{center}
\LARGE
\textbf{Explicit solution for a two--phase fractional Stefan problem with a heat flux condition at the fixed face}
\end{center}
                   \begin{center}
                  {\sc  Sabrina D. Roscani and Domingo A. Tarzia}\\
 CONICET - Depto. Matem\'atica,
FCE, Univ. Austral,\\
 Paraguay 1950, S2000FZF Rosario, Argentina \\
(sabrinaroscani@gmail.com, dtarzia@austral.edu.ar)
                   \vspace{0.2cm}

       \end{center}

\small

\noindent \textbf{Abstract: }

A generalized Neumann solution for the two-phase fractional Lam\'e--Clapeyron--Stefan problem for a semi--infinite material with constant initial temperature and a particular heat flux condition at the fixed face is obtained, when a restriction on data is satisfied. The fractional derivative in the Caputo sense of order $\al \in (0,1)$ respect on the temporal variable is considered in two governing heat equations and in one of the conditions for the free boundary. Furthermore, we find a relationship between this fractional free boundary problem and another one with a constant temperature condition at the fixed face   and based on that fact, we obtain an inequality for the coefficient which characterizes the fractional phase-change interface obtained in  Roscani--Tarzia, Adv. Math. Sci. Appl., 24 (2014), 237-249. We also recover the restriction on data and the classical Neumann solution, through the error function, for the classical two-phase Lam\'e-Clapeyron-Stefan problem for the case $\al=1$.\\

\noindent \textbf{Keywords:} Caputo Fractional Derivative, Lam\'e--Clapeyron--Stefan Problem, Neumann  Solutions,  heat flux  boundary condition, temperature boundary condition.\\

\section{Introduction}
\label{intro}
In the last decades the fractional diffusion equation has been  extensively studied \cite{Eidelman-Kochubei-LIBRO,LuMaPa-TheFundamentalSolution,FM-Libro,Povstenko,Pskhu-Libro,Pskhu:2009,SaYa:2011} and in the recent years some works on fractional free boundary problems (that is, free boundary problems where a fractional derivative is involved),  were published \cite{At:2012,BlKl:2015,CeTa:2017,JiMi:2009,RaSi:2017,RoSa:2013,RoTa:2014,Ta:2015,Voller:2014,VoFaGa:2013}. In particular, in \cite{KhoZaFe:2003}, the classical Lam\'e--Clapeyron--Stefan problem was studied by using the fractional derivative of order $\frac{1}{2}$.\\
 Recall that free boundary problems for the one--dimensional classical heat equation are problems linked to the processes of melting and freezing, which have a latent heat condition at the solid-liquid interface connecting the velocity of the free boundary and the heat fluxes of the two temperatures corresponding to the solid and liquid phases. This kind of problems are known in the literature as Stefan problems or, more precisely, as Lam\'e--Clapeyron--Stefan problems. 
We remark that the first work on phase--change problems was done by G. Lam\'e and B.P. Clapeyron in 1831 \cite{Lame:1831}  by studying the solidification of the Earth planet, which has been missing in the scientific literature for more than a century. Next, sixty years later, the phase--change problem was continued by J. Stefan through several works around year 1890 \cite{Stefan:1889} by studying the melting of the polar ice. For this reason, we call these kind of problems as Lam\'e--Clapeyron--Stefan problems or simply by Stefan problems.\\  
	Nowadays, there exist thousands of papers on the classical Lam\'e--Clapeyron--Stefan problem, for example the books \cite{Alexiades,Cannon,Crank,Elliott,Fasano:2005,Gupta,Lunardini,Rubi} and the large bibliography given in \cite{Tarzia:biblio}. Especially, a review on explicit solutions with moving boundaries was given in \cite{Tarzia}.\\
	
In this paper, a generalized Neumann solution for the two--phase fractional Lam\'e--Clapeyron--Stefan problem for a semi--infinite domain is obtained when a constant  initial data and a  Neumann boundary condition at the fixed face are considered. Recently, a generalized Neumann solution for the two--phase fractional Lam\'e-Clapeyron-Stefan problem for a semi--infinite domain with constant initial data and a Dirichlet condition at the fixed face was given in \cite{RoTa:2014}.

 So,  the classical time derivative will be  replaced by a fractional derivative in the sense of Caputo of order  $0<\al<1$, which is present in the two governing heat equations and in one of the governing conditions for the free boundary. The fractional Caputo derivative is defined in \cite{Ca:1967} as:

\begin{equation}\label{derivCaputo} D^{\alpha}f(t)=\left\{\begin{array}{lc} \frac{1}{\Gamma(1-\al)}\displaystyle\int^{t}_{0}(t-\tau)^{-\al} f'(\tau)\, \dd\tau, &  0<\al<1\\
f'(t), &   \al=1. \end{array}\right.\end{equation}

where $\Gamma$ is the Gamma function defined in $\bbR^+$ by the following expression:
$$ \Gamma(x)=\int_0^\infty t^{x-1}e^{-t}\, \dd t. $$ 

 It is known that the fractional Caputo derivative verifies that \cite{Kilbas}:\\
 For every $b\in \bbR^+$, 
\begin{equation}\label{Caputo-lineal}
D^{\alpha} \text{ is a linear operator in } W^1(0,b)=\left\{ f\,\in\, \mathcal{C}^1(0,b]\, \colon \,f'\in L^1(0,b)\right\},
\end{equation} 

\begin{equation}\label{Caputo-deriv-constantes}
D^{\alpha}(C)=0 \quad \text{for every} \, \text{constant }\, C \in \bbR
\end{equation}
and 
\begin{equation}\label{Caputo-deriv-t^al}
D^{\al}(t^\beta)=\frac{\G(\be+1)}{\G(\beta-\al+1)}t^{\beta -\al} \quad \text{for every} \, \text{constant }\, \beta >-1.\end{equation}

Now we define the two functions (Wright and Mainardi functions) which are very important in order to obtain the explicit solution given in the following Sections.\\

The Wright function is defined in \cite{Wr:1933} as:
\begin{equation}\label{W}
W(x;\rho;\be)=\sum^{\infty}_{n=0}\frac{x^{n}}{n!\G(\rho n+\be)}, \qquad x\in \bbR, \quad \rho>-1, \quad \be \in \bbR
\end{equation}

and the Mainardi function, which is a particular case of the Wright functions, is defined in \cite{GoLuMa:1999} as:
\begin{equation}\label{M}
M_\rho (x)= W(-x,-\rho,1-\rho)=\sum^{\infty}_{n=0}\frac{(-x)^n}{n! \G\left( -\rho n+ 1-\rho \right)}, \quad x\in \bbR, \, 0<\rho<1. \end{equation}

\begin{propo}\label{Propiedades Wright}Some basic properties of the Wright function are the following:
\begin{enumerate}
\item \label{deriv W} {\rm \cite{Kilbas}} The Wright function (\ref{W}) is a differentiable  function for every $\rho>-1$, $\be \in \bbR$ such that 
\begin{equation*}\frac{\p W}{\p x}\left(x;\rho;\be\right)=W\left(x;\rho;\be +\rho\right).
 \end{equation*}
\item \label{casos part W}{\rm\cite{RoSa:2013}}
$\displaystyle\lim_{\al\rightarrow 1^-}W\left(-x;-\frac{\al}{2};1\right)=W\left(-x;-\frac{1}{2};1\right)={\rm erfc}\left(\frac{x}{2}\right)$, and  \linebreak $\displaystyle\lim_{\al\rightarrow 1^-} 1- W\left(-x;-\frac{\al}{2};1\right)= 1- W\left(-x;-\frac{1}{2};1\right)=  {\rm erf}\left(\frac{x}{2}\right)$.
\item \label{deriv frac W} {\rm \cite{Pskhu-Libro}} For all $\al,c \in \bbR^+$, $\rho \in (0,1)$, $\beta \in \bbR$ we have
$$ D^\al\left(x^{\be-1}W(-cx^{-\rho},-\rho,\be) \right)=x^{\be-\al-1}W(-cx^{-\rho},-\rho,\be-\al).
$$
\item  \label{decrecientes}{\rm \cite{RoSa:2013}} For every $\al\in (0,1)$,  $W\left(-x,-\frac{\al}{2},1\right)$ is a positive and strictly decreasing function in  $\bbR^+$ such that  $0<W\left(-x,-\frac{\al}{2},1\right)< 1. $

\item\label{lim W} {\rm \cite{Wr1:1934}} For every $\al\in (0,1)$, and $ \beta>0$,    
\begin{equation} \displaystyle\lim_{x \rightarrow \infty } W\left(-x,-\frac{\al}{2},\beta\right)=0.\end{equation}
\end{enumerate}
\end{propo}

In \cite{Tar:1981} the following classical phase-change problem was studied:

\noindent \textbf{Problem:} Find the free boundary $x=s(t)$, and the temperatures $T_s=T_s(x,t)$ and $T_l=T_l(x,t)$ such that the following equations and conditions are satisfied:

\begin{equation}{\label{FP1}}
\begin{array}{lll}
     (i)  &   {T_s}_t-\lambda_s^2\,{T_s}_{xx}=0, &   x>s(t), \,  t>0,\\
     (ii) &   {T_l}_t-\lambda_l^2\,{T_l}_{xx}=0, &   0<x<s(t), \,  t>0\\ 
    (iii) & s(0)=0, \\
     (iv) & T_s(x,0)=T_s(\infty,t)=T_i<T_m & x>0, \, t>0, \\
     (v)  &  T_s(s(t),t)=T_m,   &  t>0,  \\
      (vi)  &  T_l(s(t),t)=T_m,   &  t>0,  \\
     (vii) & k_s {T_s}_x(s(t),t)-k_l{T_l}_x(s(t),t)=\rho l \dot{s}(t), & t>0,\\    
        (viii) & k_l{T_l}_x(0,t)=-\frac{q_0}{t^{1/2}}, & t>0, 
                                             \end{array}
                                             \end{equation}

where $\la_s^2=\frac{k_s}{\rho c_s}$, $\la_l^2=\frac{k_l}{\rho c_l}$, $k_s$, $c_s$   and  $k_l$,$ c_l$ are the diffusion, conductivity and specific heat coefficients of the solid and liquid phases respectively, $\rho$  is the common density of mass, $l$ is the latent heat of fusion by unit of mass,  $T_i$ is the constant initial temperature,  $T_m$  is the melting temperature and $q_0$  is the coefficient which characterizes the heat flux at the fixed face $x=0$. \\
The explicit solution to Problem (\ref{FP1}) was obtained in \cite{Tar:1981} through the error function, when the following restriction is satisfied by data:  

\begin{equation}\label{classical restriction}
q_0>\frac{k_s (T_m-T_i)}{\la_s\sqrt{\pi}}.
\end{equation}

In this paper we consider a ``fractional melting problem'', of order $0<\al<1$, for the semi-infinite material  $x>0$ with an initial constant ``fractional temperature'' and a  ``fractional heat flux boundary condition'' at the face $x=0$. We will use a Caputo derivative operator, which converges to the classical derivative when $\alpha$ tends to 1. The interesting aspect is that for the limit case ($\al =1$) the results obtained for this generalization coincide with the results of the classical case. So, in view of the analogy that exists between the classical case and the fractional one, we  make an abuse of language by using terminologies such as ``fractional temperature'', ``fractional heat equation'' or ``fractional Stefan condition''. These terms go hand in hand with the generalization proposed in the sense of operators and we do not pretend to give them a   physical approach. It is necessary to point that the physical approach associated to this type of operators is of our current interest, mainly because of the mathematical coherence that the results have together with their convergence to the classical known results.  \\  
So, the problem to be studied is the following: \\

\noindent \textbf{Problem:} Find the fractional free boundary  $x=r(t)$, defined for  $t>0$, and the fractional temperature  $\Theta=\Theta(x,t)$, defined for $x>0$, $t>0$, such that the following equations and conditions are satisfied ($0<\al<1$):
 
\begin{equation}{\label{FPalpha}}
\begin{array}{lll}
     (i)  &   D^\al_t\Theta_s-\lambda_s^2\,{\Theta_s}_{xx}=0, &   x>r(t), \,  t>0,\\
     (ii) &   D^\al_t\Theta_l-\lambda_l^2\,{\Theta_l}_{xx}=0, &   0<x<r(t), \,  t>0\\ 
    (iii) & r(0)=0, \\
     (iv) & \T_s(x,0)=\T_s(\infty,t)=T_i<T_m & x>0, \, t>0, \\
     (v)  &  \T_s(r(t),t)=T_m,   &  t>0,  \\
      (vi)  &  \T_l(r(t),t)=T_m,   &  t>0,  \\
     (vii) & k_s {\T_s}_x(r(t),t)-k_l{\T_l}_x(r(t),t)=\rho l D^\al r(t), & t>0,\\    
        (viii) & k_l{\T_l}_x(0,t)=-\frac{q_0}{t^{\al/2}}, & t>0, 
                                             \end{array}
                                             \end{equation}

Note that the suffix $t$ in the operator $D^\al$ denotes that the fractional derivative is taken in the $t-variable$. \\
In Section 2, a necessary condition for the coefficient $q_0>0$, which characterizes the fractional heat flux boundary condition at the face $x=0$, is obtained in order to have an instantaneous two-phase fractional Lam\'e--Clapeyron--Stefan problem.\\ 
In Section 3, we give a sufficient condition for the coefficient   $q_0>0$   (which coincides with the necessary condition for $q$  given in Section 2) in order to obtain a generalized Neumann solution for the two--phase fractional Lam\'e--Clapeyron--Stefan problem (\ref{FPalpha}) for a semi--infinite material with a constant initial condition and a fractional heat flux boundary condition at the fixed face  $x=0$. This solution is given as a function of the Wright and Mainardi functions.\\
Moreover, when  $\al=1$, we recover the Neumann solution, through the error function, for the classical two--phase Lam\'e--Clapeyron--Stefan problem given in \cite{Tar:1981}, when an inequality for the coefficient that characterizes the heat flux boundary condition is satisfied.\\ 
In Section 4, we consider  two two-phase fractional Lam\'e--Clapeyron--Stefan problems having a fractional heat flux and a fractional temperature boundary conditions on the fixed face $x=0$, respectively and a possible equivalence between then is analyzed.\\ 
In Section 5, an inequality for the coefficient which characterizes the free boundary of the two-phase fractional Lam\'e--Clapeyron--Stefan problem with a fractional temperature boundary condition   given recently in \cite{RoTa:2014}, is also obtained.\\
In Section 6, we recover the results obtained in \cite{RoSa:2013} for the one--phase fractional Lam\'e--Clapeyron--Stefan problem as a particular case of the present work (see Sections 3 and 4).

\section{Necessary condition to obtain an instantaneous two--phase fractional Stefan  problem with a heat flux boundary condition at the fixed face}
\label{sec:1}
In order to obtain a necessary condition for data to have an instantaneous phase-change process for problem (\ref{FPalpha}) we consider the following fractional heat conduction problem of order  $0<\al<1$ for the solid phase in the first quadrant with an initial constant temperature and a heat flux boundary condition at $x=0$:

\begin{equation}{\label{PFlujoIC}}
\begin{array}{lll}
     (i)  &   D^\al_t\Theta -\lambda_s^2\,{\Theta}_{xx}=0, &   x>0, \,  t>0,\\
     (ii) &   \Theta(x,0)=T_i, &   x>0, \\ 
   (iii) & k_s {\T}_x(0,t)=-\frac{q_0}{t^{\al/2}}, & t>0. \\
                                            \end{array}
                                           \end{equation}

\begin{lema}  We have:
\begin{enumerate}
\item  The solution of the fractional heat problem (\ref{PFlujoIC}) is given by 
\begin{equation}\label{sol-PFlujoIC} \T(x,t)=T_i+\frac{q_0 \la_s \G(1-\al/2)}{k_S} W\left(-\frac{x}{\la_s t^{\al/2}},-\frac{\al}{2},1\right), \quad x>0, \, t>0. 
\end{equation}

\item If the coefficient $q_0$ satisfies the inequalities
\begin{equation}\label{cond sin cambio de fase} 
0<q_0\leq \frac{k_s(T_m-T_i)}{ \la_s \G(1-\al/2)},
\end{equation}
then  problem (\ref{FPalpha}) is only a fractional heat conduction problem for the initial solid phase. By the contrary, if
\begin{equation}\label{cond para cambio de fase}
 q_0> \frac{k_s(T_m-T_i)}{ \la_s \G(1-\al/2)}, 
\end{equation}
then (\ref{cond para cambio de fase}) is a necessary condition for data which ensures an instantaneous fractional phase--change problem (\ref{FPalpha}).
 \end{enumerate}
\end{lema}
\begin{proof}
\begin{enumerate}\item From Proposition \ref{Propiedades Wright} items \ref{deriv W} and  \ref{deriv frac W}, and properties (\ref{Caputo-lineal}) and (\ref{Caputo-deriv-constantes}), we can state that 
\begin{equation}\label{Theta_a_b}
\T(x,t)=a+b\left[1-\Wlams\right], \qquad x>0,\, t>0,
\end{equation}
is a solution to the fractional diffusion equation $(\ref{PFlujoIC}-i)$, where $a$ and $b$ are two constants to be determined. \\
Taking the derivative of (\ref{Theta_a_b}) with  respect to $x$, by using Proposition \ref{Propiedades Wright} item \ref{deriv W},  we get
\begin{equation}\label{theta_x}
 \T_x(x,t)=\frac{b}{\la_st^{\al/2}}M_{\al/2}\left(\frac{x}{\la_s t^{\al/2}}\right). 
 \end{equation}
From conditions $(\ref{PFlujoIC}-ii)$ and  $(\ref{PFlujoIC}-iii)$, and being $M_{\al/2}(0)=\frac{1}{\G(1-\al/2)}$ and \linebreak $W\left(0,-\frac{\al}{2},1\right)=1$, we obtain that
\begin{equation}\label{a-b}
a=T_i+\frac{q_0\la_s}{k_s}\G(1-\al/2), \quad b=-\frac{q_0 \la_s \G(1-\al/2)}{k_S},
\end{equation}
 that is, we obtain the expression (\ref{sol-PFlujoIC})  as a solution to problem (\ref{PFlujoIC}).
 \item From Proposition \ref{Propiedades Wright} items \ref{decrecientes} and \ref{lim W}, it results that function (\ref{sol-PFlujoIC}) 
 $$\T(x,t)=T_i+\frac{q_0 \la_s \G(1-\al/2)}{k_S} W\left(-\frac{x}{\la_s t^{\al/2}},-\frac{\al}{2},1\right)$$
 is a decreasing function in the variable $x$ for every $t \in \bbR^+$ such that $\T(\infty,t)=T_i$ is a constant for all $t>0$. Therefore problem (\ref{FPalpha}) has an instantaneous fractional phase--change problem if and only if  the constant temperature at the boundary $x=0$  is greater than the melting temperature $T_m$, that is  if and only if 
 $$T_i+\frac{q_0\la_s}{k_s}\G(1-\al/2)>T_m,$$ 
 which is equivalent to have that inequality (\ref{cond para cambio de fase}) holds.
  \end{enumerate}
\end{proof}

\begin{obs} When  $\al=1$, the inequality {\rm(\ref{cond para cambio de fase})} is given by {\rm(\ref{classical restriction})}  because  \linebreak $\G\left(\frac{1}{2}\right)=\sqrt{\pi}$, which was first established in \cite{Tar:1981}.
\end{obs}

\subsection{Sufficient condition to obtain an instantaneous two--phase--fractional Stefan problem with a heat flux boundary condition at the fixed face}
\label{sec:2}
In this Section, we study a two--phase Lam\'e--Clapeyron--Stefan problem for the time fractional diffusion equation, of order $0<\al<1$, with an initial constant temperature and a heat flux boundary condition at the face $x=0$  given by the differential equations and initial and boundary conditions given in problem (\ref{FPalpha}). Taking into account the result in the previous Section and the method developed in \cite{RoTa:2014}, an explicit solution to problem (\ref{FPalpha}) can be obtained. In fact, we have the following result:

\begin{propo}\label{exist-sol-FPalpha} Let $T_i<T_m$ be. If the coefficient $q_0$  satisfies the inequality (\ref{cond para cambio de fase}) then there exists an instantaneous fractional phase-change (melting) process and the problem (\ref{FPalpha}) has the generalized Neumann explicit solution given by:
\begin{equation}\label{Theta_l}
\T_l(x,t)=T_m+\frac{q_0\la_l \G(1-\al/2)}{k_l}\left[ \Wlaml -W\left(-\la_l \mu_\al, -\frac{\al}{2},1\right)\right]
\end{equation}
\begin{equation}\label{Theta_s}
\T_s(x,t)=T_i+(T_m-T_i)\frac{\Wlams}{W\left(-\mu_\al, -\frac{\al}{2},1\right)}
\end{equation}
\begin{equation}\label{r}
r(t)=\mu_\al \la_s t^{\al/2}
\end{equation}
where the coefficient  $\mu_\al>0$ is a  solution of the following equation:
\begin{equation}\label{eq mu-1}
G_\al(x)=\frac{\G\left(1+\frac{\al}{2}\right)}{\G\left(1-\frac{\al}{2}\right)}x, \qquad x>0
\end{equation}
with
\begin{equation}\label{G-al}
G_\al(x)=\frac{q_0\la_l \G(1-\al/2)}{\rho l \la_s}M_{\al/2}(\la x)-\frac{k_s (T_m-T_i)}{\rho l \la_s^2}F_{2\al}(x),
\end{equation}
\begin{equation}\label{F-2al}
F_{2\al}(x)=\frac{M_{\al/2}(x)}{W\left(-x,-\frac{ \al}{2}, 1 \right)}
\end{equation}
and 
\begin{equation}\label{la}
\la=\frac{\la_s}{\la_l}>0.
\end{equation}
\end{propo}

\proof
Following  \cite{RoTa:2014}, we propose the following solution:
\begin{equation}\label{Theta_l-AB}
\T_l(x,t)=A+B\left[1- \Wlaml \right]
\end{equation}
\begin{equation}\label{Theta_s-CD}
\T_s(x,t)=C+D\left[1- \Wlams \right]
\end{equation}
\begin{equation}\label{r-mu}
r(t)=\mu \la_s t^{\al/2}
\end{equation}
where the coefficients $A, B, C, D$ and $\mu$  are constants and must be determined. According to the results in the previous Section and the linearity of the fractional derivative, functions $\T_s$ and $\T_l$ are solutions of the fractional diffusion equations $(\ref{FPalpha}-i)$ and $(\ref{FPalpha}-ii)$, respectively. Starting from conditions $(\ref{FPalpha}-vi)$ and $(\ref{FPalpha}-viii)$, we obtain the following system of two equations:

\begin{equation}\label{sis1-1}
T_m=\T_l(r(t),t)=A+B\left[ 1- W\left(\mu \la , -\frac{\al}{2},1 \right)\right]
\end{equation}

\begin{equation}\label{sis1-2}
-\frac{q_0}{t^{\al/2}}=k_l\T_{l_\infty}(0,t)=\frac{B k_l}{\la_l t^{\al/2}}M_{\al/2}(0),
\end{equation}

from which we obtain:
\begin{equation}\label{sol-sis1}
A=T_m+\frac{q_0 \la_l \G\left(1-\frac{\al}{2}\right)}{k_l}\left[1-W\left(-\mu \la_l, -\frac{\al}{2},1 \right)\right], \quad \quad B=\frac{q_0\la_l \G\left(1-\frac{\al}{2}\right)}{k_l}.
\end{equation}
Then, the fractional temperature of the liquid phase is given by (\ref{Theta_l}). 


From conditions $(\ref{FPalpha}-iv)$ and $(\ref{FPalpha}-v)$ we have the system of equations:

\begin{equation}\label{sis2-1}
T_i=\T_s(x,0)=C+D,
\end{equation}

\begin{equation}\label{sis2-2}
T_m=\T_s(r(t),t)=C+D\left[1 -W\left(-\mu, -\frac{\al}{2},1\right)\right],
\end{equation}
and then we have:
\begin{equation}\label{sol-sis2}
C=T_i+\frac{T_m-T_i}{W\left(-\mu, -\frac{\al}{2},1\right)}, \quad   D=-\frac{T_m-T_i}{W\left(-\mu, -\frac{\al}{2},1\right)}.
\end{equation}
Therefore, the fractional temperature of the solid phase is given by (\ref{Theta_s}).\\

In order to determine the coefficient $\mu >0$  we must consider the fractional Lam\'e--Clapeyron--Stefan condition $(\ref{FPalpha}-vii)$ which, taking into account Proposition \ref{Propiedades Wright}  and (\ref{Caputo-deriv-t^al}), gives us the equation $(\ref{eq mu-1})$.\\
It was proved in \cite{RoTa:2014} that $F_{2\al}(+\infty)=+\infty$, then the function $G_\al=G_\al(x)$, defined by (\ref{G-al}), has the following properties:

\begin{equation}\label{prop-G_al}
G_\al(0^+)=\frac{q_0\la_l }{\rho l \la_s}-\frac{k_s (T_m-T_i)}{\rho l \la_s^2 \G(1-\frac{\al}{2})}, \qquad G_\al(+\infty)=-\infty.
\end{equation}

From the continuity of $G_\al$ (due to Proposition \ref{Propiedades Wright}  and \ref{decrecientes}) and (\ref{prop-G_al})), it yields that equation (\ref{eq mu-1}) has a solution $\mu_\al >0$ if $G_\al(0^+)>0$ which is verified under condition (\ref{cond para cambio de fase}). Then, the solution $\left\{(\ref{Theta_l}), (\ref{Theta_s}),(\ref{r})\right\}$ holds.

\endproof

\begin{obs}
The solution of the equation (\ref{eq mu-1}) will be unique if we can prove that function  $G_{\al}$ is a strictly decreasing function in  $\bbR^+$, or equivalently if we can prove that  function $F_{2\al}$  is an increasing function in $\bbR^+$  (taking into account that function $M_{\al/2}$  is a decreasing function). 

 Function  $F_{2\al}(x)=\frac{M_{\al/2}( x)}{W\left(-x;-\frac{\al}{2}, 1\right)}$ is a
continuous positive function, which is a quotient of two decreasing functions. Some graphics are presented below:

\begin{figure}[h]
\centering
\mbox {\subfigure[{\footnotesize $F_{2\al}$ for $\al=1/16, 1/8, 1/4, 3/8$ and $1/2.$ }
]{\epsfysize=55mm \epsfbox {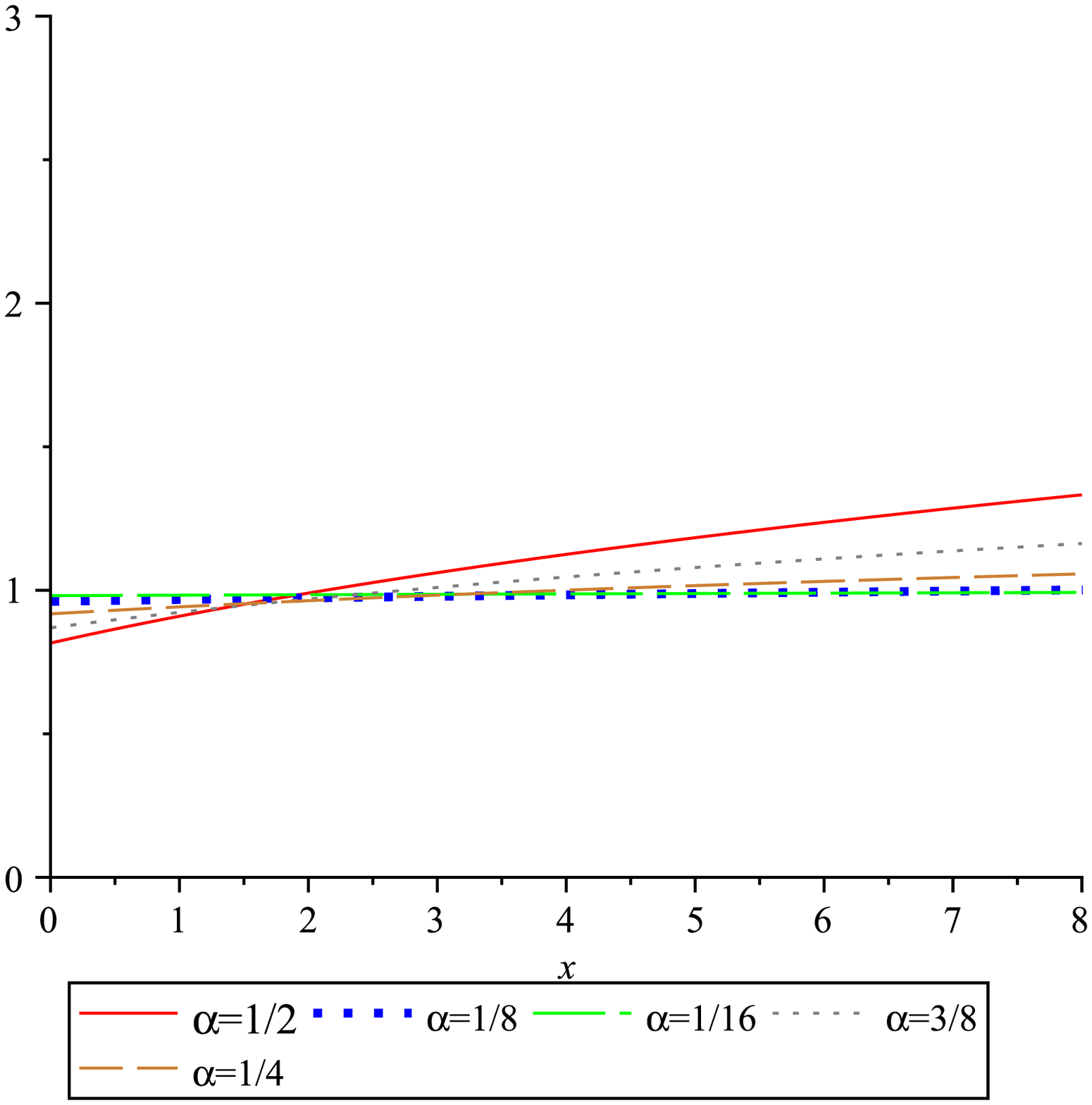
}} \quad \quad \quad
\subfigure[{\footnotesize $F_{2\al}$ for $\al=1/2, 5/8, 7/8, 3/4$ and $15/16.$} 
] {\epsfysize=55mm \epsfbox{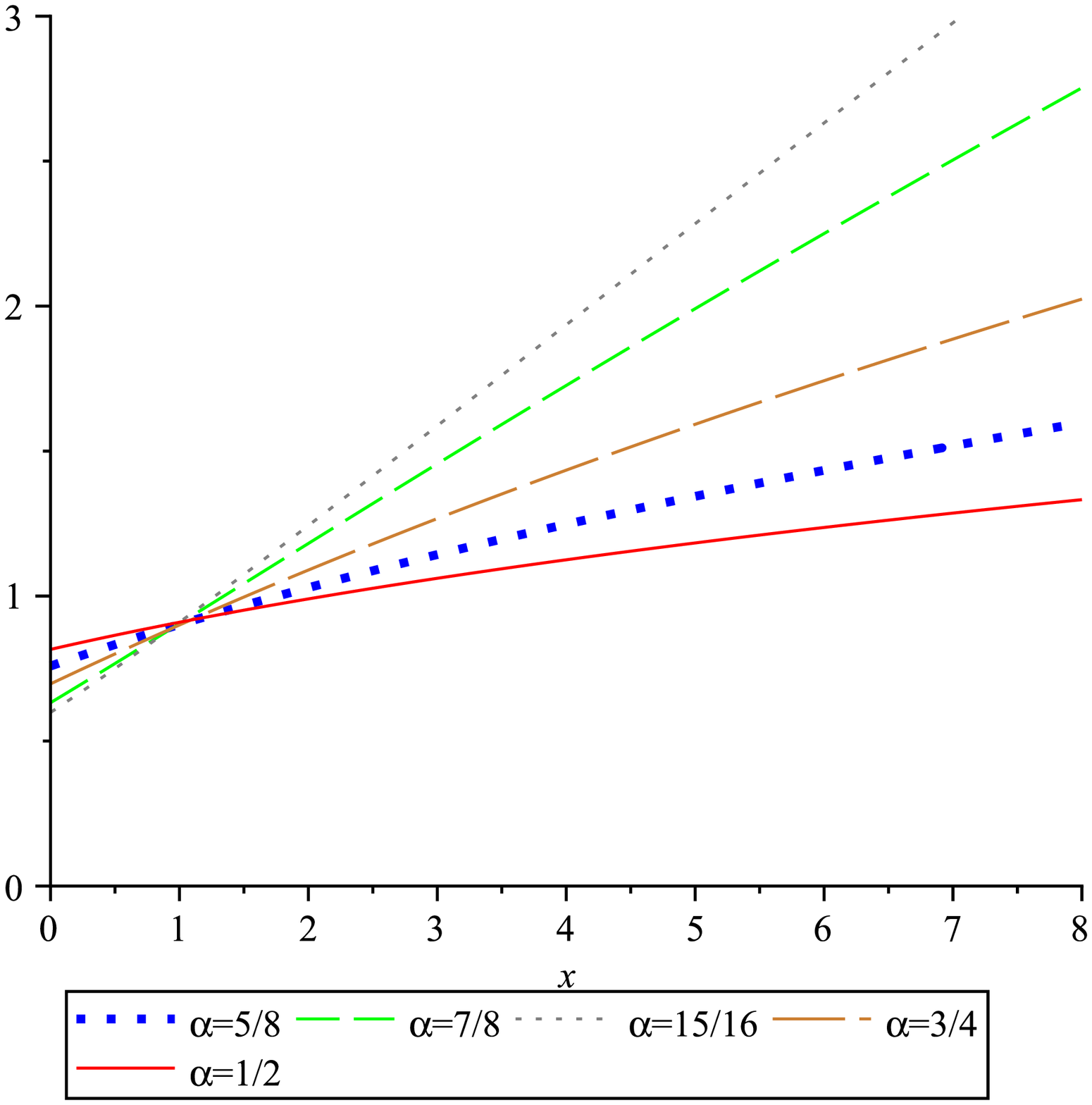
}} }
\end{figure}

We can see in the graphics that $F_{2\al}$  is an increasing function for every chosen  parameter. We wonder if this situation is true for every $\al \in (0,1)$. A simple computation gives that $F_{2\al}$  is an increasing function if and only if 
\begin{equation}\label{F'2}
\left[M_{\al/2}( x)\right]^2-W\left(-x;-\frac{\al}{2}, 1\right)\cdot W\left(-x;-\frac{\al}{2}, 1-\alpha\right)>0.
\end{equation}

We have proved in \cite{RoTa:2017-TwoDifferent} that for every $x>0$,
\begin{equation}
\G(1-\al)W\left(-x;-\frac{\al}{2}, 1-\alpha \right)>\G\left(1-\frac{\al}{2}\right)M_{\al/2}( x)> W\left(-x;-\frac{\al}{2}, 1\right)>0,\end{equation}
but this is not a sufficient condition to prove (\ref{F'2}). \\
Also, the inequality (\ref{F'2}) is a Tur\'an-type inequality for Wright functions of parameter $-\frac{\al}{2} \in (-1,0)$. An analogue result for Wright functions with positive parameter was proved in \cite{Me:2017}, that is, it was proved that 
$$  \left[W\left(x;\al, \beta+\al\right)\right]^2- W\left(x;\al, \beta\right)W\left(x;\al, \beta+2\al\right)\geq 0, \, \forall \, x>0, \al>0, \, \be>0.$$
So we state the following conjecture:
\end{obs}

\begin{conj} The function $F_{2\al} \colon \bbR^+ \rightarrow \bbR^+$ such that $F_{2\al}(x)=\frac{M_{\al/2}( x)}{W\left(-x;-\frac{\al}{2}, 1\right)}$  is an increasing function with $F_{2\al}(0^+)=\frac{1}{\G\left(1-\frac{\al}{2}\right)}$ and $F_{2\al}(+\infty)=+\infty$. 
\end{conj}

\begin{teo}  Let $T_i<T_m$  be. If the coefficient $q_0$  satisfies the inequality (\ref{cond para cambio de fase}), under the assumption of Conjecture 1,  then $\left\{(\ref{Theta_l}),(\ref{Theta_s}), (\ref{r})\right\} $ is the unique generalized Nuemann similarity--solution to the free boundary problem (\ref{FPalpha}), where $\mu_\al$ is the unique solution to equation (\ref{eq mu-1}).
\end{teo}

\begin{obs} If $(x,t)$ is in the liquid face  ($0<x<s(t), t>0$), then $0<x<\la_s\mu t^{\al/2}$, or equivalently  $0<\frac{x}{t^{\al/2}}<\la_s\mu $.  Multiplying by $\la_l$ gives $0<\frac{x}{\la_l t^{\al/2}}<\la\mu $.  Then, from  Proposition \ref{Propiedades Wright} item \ref{decrecientes} it follows that $ \Wlaml -W\left(-\la_l \mu_\al, -\frac{\al}{2},1\right)>0 $ and therefore   the explicit temperature of the liquid phase corresponding to problem (\ref{FPalpha}) satisfy the following inequality

\begin{equation}\label{theta_l>T_m}
\begin{split}
\T_l(x,t)& =T_m+\frac{q_0\la_l \G(1-\al/2)}{k_l}\left[ \Wlaml -W\left(-\la_l \mu_\al, -\frac{\al}{2},1\right)\right]\\
 & >T_m,  \qquad 0<x<r(t), \quad t>0
 \end{split}
\end{equation}

 Analogously the explicit temperature of the solid phase corresponding to problem (\ref{FPalpha}) satisfy the following inequality

\begin{equation}\label{theta_s>T_i}
T_i<\T_s(x,t)<T_m, \qquad x>r(t), \quad t>0.
\end{equation}
\end{obs}

\begin{propo} Let  $T_i<T_m$ be. By considering $\al =1$ in Proposition \ref{exist-sol-FPalpha},  we recover the classical Neumann explicit solution and the inequality (\ref{classical restriction}) for the coefficient which characterized the heat flux at  $x=0$ obtained  in \cite{Tar:1981}.
\end{propo}

\proof 
As it was said in Remark 1, the inequality (\ref{classical restriction})  is recovered because 
$\G(1/2)=\sqrt{\pi}$. By the other side, 
\begin{equation}\label{Theta_l-1}
\T_l(x,t)=T_m+\frac{q_0\la_l \G(1-\al/2)}{k_l}\left[W\left(-\frac{x}{\la_l t^{1/2}}, -\frac{1}{2},1\right) -W\left(-\la_l \mu_1, -\frac{\al}{2},1\right)\right]
\end{equation}
\begin{equation}\label{Theta_s-1}
\T_s(x,t)=T_i+(T_m-T_i)\frac{W\left(-\frac{x}{\la_2 t^{1/2}}, -\frac{1}{2},1\right)}{W\left(-\mu_1, -\frac{1}{2},1\right)}
\end{equation}
\begin{equation}\label{r-1}
r(t)=\mu_1 \la_s t^{1/2}
\end{equation}
where the coefficient  $\mu=\mu_1>0$ is the  solution of equation:
\begin{equation}\label{eq mu-1-1}
G_1(x)=\frac{\G\left(3/2\right)}{\G\left(1/2\right)}x, \qquad x>0
\end{equation}
with
\begin{equation}\label{G-al}
G_1(x)=\frac{q_0\la_l \G(1/2)}{\rho l \la_s}M_{1/2}(\la x)-\frac{k_s (T_m-T_i)}{\rho l \la_s^2}F_{2}(x),
\end{equation}
\begin{equation}\label{F-2al}
F_{2}(x)=\frac{M_{1/2}(x)}{W\left(-x,-\frac{ 1}{2}, 1 \right)}
\end{equation}
and 
\begin{equation}\label{la}
\la=\frac{\la_s}{\la_l}>0.
\end{equation}
Taking into account that $\G(3/2)=\frac{\sqrt{\pi}}{2}$, $M_{1/2}(x)=e^{-(x/2)^2}$ and that \linebreak  $W\left(-x,-\frac{ 1}{2}, 1 \right)={\rm erfc}\left(\frac{x}{2}\right)$ (see \cite{RoSa:2013}), it results that 
\begin{equation}\label{theta_s_lim}
\T^1_s(x,t) =T_i+(T_m-T_i)\frac{{\rm erfc}\left(\varsims\right)}{{\rm erfc}\left(\frac{\mu_1}{2}\right)},
\end{equation}  
\begin{equation}\label{theta_l_lim}
\T^1_l(x,t) = T_m+\frac{q_0\la_l \sqrt{\pi}}{k_l}\left[ {\rm{erfc}}\left( \frac{x}{\lambda_l t^{1/2}}\right)- {\rm{erfc}}\left(\frac{\lambda \mu_1}{2}\right) \right],
\end{equation}  
\begin{equation}\label{r_lim}
r_1(t)=\mu_1\la_s\sqrt{t},
\end{equation}  
where $\mu_1>0$ is the solution of the equation:

\begin{equation}\label{eq mu clasica-1}
\frac{q_0}{\rho l \la_s}{\rm exp}\left(-\frac{\la^2x^2}{4}\right)-\frac{k_s(T_m-T_i)}{\rho l \la_s^2 \sqrt{\pi}}\frac{{\rm exp}\left(-\frac{x^2}{4}\right)}{{\rm erfc}\left(\frac{x}{2}\right)}=\frac{x}{2}, \quad x>0
\end{equation}
or equivantely, $\frac{\mu_1}{2}$ is a solution of the equation:
\begin{equation}\label{eq mu clasica-2}
\frac{q_0}{\rho l \la_s}{\rm exp}\left(-\la^2x^2\right)-\frac{k_s(T_m-T_i)}{\rho l \la_s^2 \sqrt{\pi}}\frac{{\rm exp}\left(-x^2\right)}{{\rm erfc}\left(x\right)}=x, \quad x>0.
\end{equation}

Therefore, the tender $\left\{\T_s^1(x,t), \T_l^1(x,t), r_1(t)\right\}$, where  $\mu_1/2$ is    the solution of the equation (\ref{eq mu clasica-2}), is the solution of the problem (\ref{FP1}) given in \cite{Tar:1981}. 

\endproof

\begin{teo}. Let  $T_i<T_m$ be. If the coefficient  $q_0$ satisfies the inequality (\ref{cond para cambio de fase}) and the Conjecture 1 holds,  then the similarity-solution to the problem (\ref{FPalpha}) converges to the similarity-solution to the classical Lam\'e-Clapeyron-Stefan problem (\ref{FP1}) when  $\al \rightarrow 1^-$. 
\end{teo}

\subsection{Two--phase fractional Stefan problems with a heat flux and a temperature boundary condition at the fixed face admitting the same similarity solution}
\label{sec:3}

Let  $T_i<T_m$ be. If the coefficient $q_0$  satisfies the inequality (\ref{cond para cambio de fase}), then the solution of the problem (\ref{FPalpha} )  is given by (\ref{Theta_l}), (\ref{Theta_s}) and (\ref{r}) where $\mu_\al$  is a solution of the equation (\ref{eq mu-1}). In this case, we can compute the liquid temperature $\Theta_l$  at the fixed face $x=0$, which is given by:
  
\begin{equation}\label{Theta_l(0.t)}
\T_l(0^+,t)=T_m+\frac{q_0\la_l \G(1-\al/2)}{k_l}\left[1-W\left(-\la \mu_\al; -\frac{\al}{2},1\right)\right]>T_m, \quad \forall \, t>0.
\end{equation}
Since this temperature is greater than $T_m$ the melting temperature   and it is constant for all positive time, we can consider  the following fractional free boundary problem:\\

\textbf{Problem: }  Find the free boundary  $x=s(t)$, defined for $t>0$, and the temperature  \linebreak $T=T(x,t)$, defined for $x>0, t>0$  such that the following equations and conditions are satisfied ($0<\al<1$):

\begin{equation}\label{TPalpha}
\begin{array}{lll}
     (i)  &   D^\al_tT_s-\lambda_s^2\,{T_s}_{xx}=0, &   x>s(t), \,  t>0,\\
     (ii) &   D^\al_tT_l-\lambda_l^2\,{T_l}_{xx}=0, &   0<x<s(t), \,  t>0\\ 
    (iii) & s(0)=0, \\
     (iv) & T_s(x,0)=T_s(+\infty,t)=T_i<T_m & x>0, \, t>0, \\
     (v)  &  T_s(s(t),t)=T_m,   &  t>0,  \\
      (vi)  &  T_l(s(t),t)=T_m,   &  t>0,  \\
     (vii) & k_s {T_s}_x(s(t),t)-k_l{T_l}_x(s(t),t)=\rho l D^\al_t s(t), & t>0,\\    
        (viii) & T(0,t)=T_0, & t>0, 
                                             \end{array}
\end{equation}
where the imposed temperature $T_0$ at the fixed face  $x=0$ is greater than the melting temperature, that is  $T_0>T_m$. The problem (\ref{TPalpha}) was recently solved in \cite{RoTa:2014} and the solution is given by:

\begin{equation}\label{T_l}
\begin{split}
T_m<T_l(x,t)=T_m+(T_0-T_m)\frac{\Wlaml -W\left(-\la \xi_\al, -\frac{\al}{2},1\right) }{1-W\left(-\la \xi_\al, -\frac{\al}{2},1\right)}\\
=T_0-(T_0-T_m)\frac{1-\Wlaml }{1- W\left(-\la \xi_\al, -\frac{\al}{2},1\right)}, \quad 0<x<s(t), \quad t>0
\end{split}
\end{equation}
\begin{equation}\label{T_s}
\begin{split}
T_i<T_s(x,t)=T_i+(T_m-T_i)\frac{\Wlams}
{W\left(-\xi_\al, -\frac{\al}{2},1\right)}=\\
=T_m-(T_m-T_i)\left[1-\frac{\Wlams}{W\left(-\xi_\al;\frac{\al}{2},1\right)}\right]<T_m, \quad x>s(t), \, t>0
\end{split}
\end{equation}
\begin{equation}\label{s}
s(t)=\xi_\al \la_s t^{\al/2}
\end{equation}
where the coefficient  $\xi=\xi_\al>0$ is a  solution of the following equation:
\begin{equation}\label{eq xi}
F_\al(x)=\frac{\G\left(1+\frac{\al}{2}\right)}{\G\left(1-\frac{\al}{2}\right)}x, \qquad x>0
\end{equation}
with
\begin{equation}\label{F-al}
F_\al(x)=\frac{k_l(T_0-T_m)}{\rho l \la_s\la_l}F_{1\al}(\la x)-\frac{k_s (T_m-T_i)}{\rho l \la_s^2}F_{2\al}(x),
\end{equation}
\begin{equation}\label{F-1al}
F_{1\al}(x)=\frac{M_{\al/2}(x)}{1-W\left(-x,-\frac{ \al}{2}, 1 \right)}
\end{equation}
and $F_{2\al}$ was defined in (\ref{F-2al}).

\begin{propo}Let $T_i<T_m$  be. If the coefficient $q_0$  satisfies the inequality (\ref{cond para cambio de fase})  then both  free boundary problems (\ref{FPalpha}) and (\ref{TPalpha}) with   data $T_0$  given by
\begin{equation}\label{T_0}
T_0=T_m+\frac{q_0\la_l \G(1-\al/2)}{k_l}\left[1-W\left(-\la \mu_al;-\frac{\al}{2},1\right)\right]
\end{equation}
admit the same similarity solutions.
\end{propo}

\proof 
Let  $T_i<T_m$ be. If the coefficient  $q_0$ satisfies the inequality (\ref{cond para cambio de fase}) then the solution of the free boundary problem (\ref{FPalpha}) is given by (\ref{Theta_l})--(\ref{r}), where the coefficient $\mu_\al$  is a solution of equation (\ref{eq mu-1}). In this case, the temperature at the fixed face $x=0$  is given by (\ref{Theta_l(0.t)}) and therefore we can now consider the free boundary problem (\ref{TPalpha}) with data $(\ref{TPalpha}-vii)$ at the fixed face  $x=0$, where $T_0$  is defined by (\ref{Theta_l(0.t)}). Note that 
\begin{equation}\label{equiv}
\begin{split}
F_\al(\mu_\al)&=\frac{k_l(T_0-T_m)}{\rho l \la_s\la_l}F_{1\al}(\la \mu_\al)-\frac{k_s (T_m-T_i)}{\rho l \la_s^2}F_{2\al}(\mu_\al)
\\
&=\frac{q_0 \G(1-\al/2)}{\rho l \la_s}\left[1-W\left(-\la \mu_\al;-\frac{\al}{2}, 1\right)\right]F_{1\al}(\la \mu_\al)-\frac{k_s (T_m-T_i)}{\rho l \la_s^2}F_{2\al}(\mu_\al)\\
&= \frac{q_0 \G(1-\al/2)}{\rho l \la_s}
 M_{\al/2}(\la \mu_\al)-\frac{k_s (T_m-T_i)}{\rho l \la_s^2}F_{2\al}(\mu_\al)=G_\al(\mu_\al)
\end{split}
\end{equation}
Then, we can affirm that $\mu_\al$ is a solution to $(\ref{eq mu-1})$ if and only if  $\mu_\al$ is a solution to $(\ref{eq xi})$.\\
Therefore, we have   solutions  given by (\ref{T_l})--(\ref{s}) and (\ref{Theta_l})--(\ref{r}) to problems (\ref{TPalpha}) and (\ref{FPalpha})  respectively, where the coefficient $\xi_\al=\mu_\al$.\\
Clarely, for every $0<\al<1$,  it results that $r(t)=s(t)$ for all $t>0$. Moreover $T_s(x,t)=\T_s(x,t)$ and  $T_l(x,t)=\T_l(x,t)$, and the thesis holds. 
\endproof

\begin{teo}  Let $T_i<T_m$  be. If the coefficient $q_0$  satisfies the inequality (\ref{cond para cambio de fase}), under the assumption of conjecture 1,  then the free boundary problem (\ref{FPalpha}) is equivalent to the free boundary problem (\ref{TPalpha}), in the sense of similarity solutions,  with data $T_0$  given by:
 \begin{equation}\label{T_0}
T_0=T_m+\frac{q_0\la_l \G(1-\al/2)}{k_l}\left[1-W\left(-\la \mu_\al;-\frac{\al}{2},1\right)\right].
\end{equation}
\end{teo}
\proof If the Conjecture 1 is true, then equation $(\ref{eq xi})$ admits a unique positive solution.

\subsection{Inequality for the coefficient which characterizes the free boundary for the two--phase fractional Stefan problem with a temperature boundary condition at the fixed face}
\label{sec:4}

Now, we consider problem (\ref{TPalpha})  with data $T_i<T_m<T_0$, whose solution given by (\ref{T_l})--(\ref{eq xi}) has been recently obtained in \cite{RoTa:2014}. 

\begin{teo} The coefficient $\xi_\al$  which characterizes the phase--change interface (\ref{s}) of the free boundary problem (\ref{TPalpha} ) verifies the inequality
 \begin{equation}\label{ineq xi}
1-W\left(-\frac{\la_s}{\la_l}\xi_\al; -\frac{\al}{2},1\right)<\frac{T_0-T_m}{T_m-T_i}\frac{k_l \la_s}{k_s \la_l}.
\end{equation}
\end{teo}

\proof If we consider the solution (\ref{T_l})-(\ref{eq xi}) of the free boundary problem  (\ref{TPalpha}) where the coefficient  $\xi_\al >0$ is a solution of the equation (\ref{eq xi}) for data $T_0>T_m$, then, by taking into account Proposition \ref{Propiedades Wright}, we have that the corresponding coefficient  $q_0$ (which characterizes the heat flux boundary condition $(\ref{FPalpha}-viii)$ on the fixed face $x=0$) is given by:
\begin{equation}\label{q0=}
q_0= \frac{T_0-T_m}{1-W\left(-\la \mu_\al; -\frac{\al}{2},1\right)} \frac{k_l}{\la_l \G(1-\al/2)}.
\end{equation}

 and then we can compute the coefficient  $q_0$. Therefore, the inequality (\ref{cond para cambio de fase}) for $q_0$  is transformed in the inequality (\ref{ineq xi}) for the coefficient $\xi_\al$  defined in \cite{RoTa:2014}, and therefore the result holds.
\endproof


\begin{obs} If we consider $\al=1$  in the inequality (\ref{ineq xi}) we obtain the inequality 
\begin{equation}{\rm erf}\left(\frac{\la_s}{\la_l}\frac{\mu_1}{2}\right)<\frac{T_0-T_m}{T_m-T_i}\frac{k_l\la_s}{k_s\la_l}  
\end{equation}
given in \cite{Tar:1981} for the  Neumann solution for the classical two--phase Stefan problem. 
\end{obs}

\section{ The One--Phase Fractional Stefan Problem}
In \cite{RoSa:2013}, the following two one-phase fractional Lam\'e--Clapeyron--Stefan problems were studied:
\begin{equation}{\label{FPalpha-one-phase}}
\begin{array}{rll}
     (i)  &   D^\al_t\Theta-\lambda^2\,{\Theta}_{xx}=0, &    0<x<r(t), \,  t>0\\ 
    (ii) & r(0)=0, \\
     (iii)  &  \T(r(t),t)=T_m,   &  t>0,  \\
     (iv) & -k{\T}_x(r(t),t)=\rho l D^\al_t r(t), & t>0,\\    
        (v) & k{\T}_x(0,t)=-\frac{q_0}{t^{\al/2}}, & t>0, 
                                             \end{array}
                                             \end{equation}

and
\begin{equation}{\label{TPalpha-one-phase}}
\begin{array}{rll}
     (i)  &   D^\al_t T-\lambda^2\,{T}_{xx}=0, &    0<x<s(t), \,  t>0\\ 
    (ii) & s(0)=0, \\
     (iii)  &  T(s(t),t)=T_m,   &  t>0,  \\
     (iv) & -k{T}_x(s(t),t)=\rho l D^\al_t s(t), & t>0,\\    
        (v) & k{T}_x(0,t)=T_0, & t>0, 
                                             \end{array}
                                             \end{equation}

where $\la^2=\frac{k}{\rho c}$. These two problems can be considered as particular cases of the free boundary problems (\ref{FPalpha}) and (\ref{TPalpha}) respectively.

\begin{coro} The results given in \cite{RoSa:2013} for the one-phase fractional Stefan problems (\ref{FPalpha-one-phase}) and (\ref{TPalpha-one-phase}) can be recovered by taking  $T_i=T_m$ in the free boundary problems (\ref{FPalpha}) and (\ref{TPalpha}) respectively.
\proof It is sufficient to observe that the inequality (\ref{cond para cambio de fase}) is automatically verified if we take $T_i=T_m$  because $q_0>0$. Then the two free boundary problems (\ref{FPalpha-one-phase}) and (\ref{TPalpha-one-phase}) are equivalents respect on similarity solutions.

\end{coro}

\section{Conclusions}
\begin{itemize}
	\item We have obtained a generalized Neumann solution for the two--phase fractional Lam\'e--Clapeyron--Stefan problem for a semi--infinite material with a constant initial condition and a heat flux boundary condition on the fixed face $x=0$, when a restriction on data is satisfied. The explicit solution is given through the Wright and Mainardi functions.
	
	\item When  $\al =1$, we have  recovered the Neumann solution through the error function for the corresponding classical two--phase Lam\'e--Clapeyron--Stefan problem given in \cite{Tar:1981}. We also recover the inequality for the corresponding coefficient that characterizes the heat flux boundary condition at  $x=0$. 
\item We have proposed a conjecture, from which it can be proved  the equivalence between the two-phase fractional Lam\'e--Clapeyron--Stefan problems with a heat flux and a temperature boundary conditions on the fixed face  $x=0$ for  similarity solutions. Moreover, an inequality for the coefficient which characterizes the free boundary given in \cite{RoTa:2014} was obtained.
\item 	We have recovered the results obtained in \cite{RoSa:2013} for the one--phase fractional Lam\'e--Clapeyron--Stefan problem as a particular case of the present work by taking $T_i=T_m$.
\end{itemize}

\section{Acknowledgements}

\noindent The present work has been sponsored by the Projects PIP N$^\circ$ 0275 from CONICET--Univ. Austral, and ANPCyT PICTO Austral N$^\circ 0090$ (Rosario, Argentina). The result of the Proposition 2 was communicated by the second author in the 1st Pan--Anamerican Congress on Computational Mechanics (Buenos Aires, 22--29 April 2015).

\end{document}